\title{Nonabelian Poincar\'{e} duality after stabilizing}
\author{
        Jeremy Miller \\
                Department of Mathematics\\
        CUNY Graduate Center\\
        365 Fifth Avenue, New York, NY
}
\date{\today}

\documentclass[12pt]{article}

\usepackage {amsmath, amscd, amsbsy, amsfonts}
\usepackage{latexsym,amssymb,mathrsfs,bbm}
\usepackage[all]{xy}
\usepackage{graphicx}
\usepackage{placeins}

\newcounter{prob}
\newcommand{\Simga}{\Sigma}
\newcommand{\Z}{\mathbb{Z}}

\newcommand{\R}{\mathbb{R}}
\newcommand{\m}{\longrightarrow}

\newcommand{\N}{\mathbb{N}_0}
\newcommand{\D}{\mathbb{D}}
\newcommand{\e}{\'{e} }
\newcommand{\p}{parallelizable }
\newcommand{\fm}{Fulton-MacPherson }

\newtheorem{theorem}{Theorem}[section]
\newtheorem{lemma}[theorem]{Lemma}
\newtheorem{proposition}[theorem]{Proposition}
\newtheorem{corollary}[theorem]{Corollary}

\newtheorem{definition}[theorem]{Definition}
\newtheorem{example}[theorem]{Example}

\newtheorem{conjecture}[theorem]{Conjecture}

\newenvironment{proof}[1][Proof]{\begin{trivlist}
\item[\hskip \labelsep {\bfseries #1}]}{\end{trivlist}}

\newcommand{\qed}{\nobreak \ifvmode \relax \else
      \ifdim\lastskip<1.5em \hskip-\lastskip
      \hskip1.5em plus0em minus0.5em \fi \nobreak
      \vrule height0.75em width0.5em depth0.25em\fi}

\begin{document}
\maketitle

\begin{abstract}
We generalize the nonabelian Poincar\'{e} duality theorems of Salvatore in \cite{Sa} and Lurie in \cite{Lu} to the case of not necessarily grouplike $E_n$-algebras (in the category of spaces). We define a stabilization procedure based on McDuff's ``bringing points in from infinity'' maps from \cite{Mc1}. For open connected parallelizable $n$-manifolds, we prove that, after stabilizing, the topological chiral homology of $M$ with coefficients in an $E_n$-algebra $A$, $\int_M A$, is homology equivalent to $Map^c(M,B^n A)$, the space of compactly supported maps to the $n$-fold classifying space of $A$.

\end{abstract}

\section{Introduction}

In this paper, we will be interested in two models of topological chiral homology. The first model that we will consider uses May's two-sided bar construction \cite{M} and was communicated to the author by Ricardo Andrade. It is a simplification of his construction from \cite{An} and is known to be equivalent to the definition of topological chiral homology introduced by Lurie in \cite{Lu}. For an explanation of this equivalence, see Remark 3.15 of \cite{Fr1}. We will also be interested in an even earlier model, Salvatore's configuration spaces of particles with summable labels \cite{Sa}. It is widely believed that this construction is equivalent to that of Lurie and Andrade. However, to the best of the author's knowledge, a proof of this equivalence has not been written up explicitly in the literature. We will prove a theorem for both models, which we call ``nonabelian Poincar\e duality after stabilizing.'' Even if one could prove that these two constructions are equivalent, it would still be interesting to have both proofs since the different nature of the proofs highlights the advantages and disadvantages of each construction. The two-sided bar construction has a natural filtration and thus one can use spectral sequence arguments to study its homology. On the other hand, for Salvatore's configuration spaces, the notion of relative configuration space is easy to define so it is possible to mimic arguments used in the 1970's to study classical configuration spaces \cite{Mc1} \cite{B}.

Topological chiral homology is a collection of constructions which take as input an $E_n$-algebra $A$ and a parallelized $n$-manifold $M$ and produces a space often denoted $\int_M A$. There is a scanning map $s: \int_M A \m Map^c(M,B^n A)$, the space of compactly supported functions from $M$ to the $n$-fold classifying space of $A$. An $E_n$-algebra $A$ is called grouplike if the induced monoid structure on $\pi_0(A)$ is a group. For grouplike $E_n$-algebras, the scanning map  $s: \int_M A \m Map^c(M,B^n A)$ is a homotopy equivalence \cite{Sa} \cite{Lu}. This fact is called nonabelian Poincar\e duality since it is equivalent to Poincar\e duality when $A=\Z$ after taking homotopy groups \cite{K3}.

If $A$ is not grouplike, then the scanning map is not a homotopy equivalence. For example, when $M=\R^n$, $\int_M A$ is homotopy equivalent to $A$ while $ Map^c(M,B^n A)$ is homotopy equivalent to $\Omega^n B^n A$. From now on, we assume that $n$ is at least $2$ (see \cite{MSe} for a discussion of the case $n=1$). At the level of $\pi_0$, the scanning map is not an isomorphism but is instead the inclusion of a commutative monoid into its Grothendieck group. While the scanning map $s$ is not a homotopy equivalence, it is a group completion. Let $\{ a_i \}$ be representatives of generators of $\pi_0(A)$ and let $m_{i}: A \m A$ be multiplication by $a_i$ maps. The group completion theorem \cite{MSe} states that the induced map $s: hocolim_{m_i} A \m \Omega^n B^n A$ is a homology equivalence. We will say that after stabilizing, $A$ is homology equivalent $\Omega^n B^n A$. We shall generalize this to arbitrary open parallelizable manifolds  $M$ and prove that, after stabilizing, $\int_M A$ is homology equivalent to $Map^c(M,B^n A)$. By open, we mean the interior of a not necessarily compact manifold $\bar M$ with non-empty boundary. We prove this result both for Andrade's two-sided bar construction model of topological chiral homology as well as Salvatore's configuration spaces of particles with summable labels. The stabilization maps $t_i:\int_M A \m \int_M A$ will be generalizations of  McDuff's ``bringing points in from infinity'' maps introduced in \cite{Mc1}. The goal of this paper is to prove the following theorem.

\begin{theorem}
Let $M$ be the interior of a connected (not necessarily compact) $n$-manifold with nonempty boundary and with $n>1$. There are stabilization maps $t_i : \int_M A \m \int_M A$ and a scanning map $s : \int_M A \m Map^c(M, B^n A)$ such that $s$ induces a homology equivalence between $hocolim_{t_i} \int_M A$ and $Map^c(M,B^n A)$. 

\label{theorem:main}
\end{theorem}

In Salvatore's model, we are able to generalize Theorem \ref{theorem:main} to not necessarily parallelizable manifolds. Since it is widely believed that the two models of topological chiral homology are equivalent, it is natural to conjecture that Theorem \ref{theorem:main} for Andrade's model can be generalized to non-parallelizable manifolds. Unfortunately however, the methods of proof of this paper do not seem equipped to  handle this generality. We also conjecture that similar theorems (nonabelian Poincar\'{e} duality after stabilizing) are true for bundles of $E_n$-algebras and $E_n$-algebras in categories other than spaces.

In Section 2, we prove Theorem \ref{theorem:main} when $\int_M A$ is the model of topological chiral homology defined using the two sided bar construction due to Andrade and in Section 3 we prove the theorem when $\int_M A$ is Salvatore's configuration space of particles with summable labels. Since Andrade's model is known to be homotopy equivalent to Lurie's model, Theorem \ref{theorem:main} will also be true when $\int_M A$ is interpreted to mean Lurie's definition of topological chiral homology. In Section 4, we make a conjecture regarding homological stability for the connected components of $\int_M A$.

\paragraph{Acknowledgments}

I would like to thank Ricardo Andrade and Jonathan Campbell for many helpful discussions and the referee for many important suggestions.

\FloatBarrier
\section{Topological chiral homology via the two-sided bar construction}

In this subsection, we will describe a definition of topological chiral homology derived from \cite{An} employing the monadic two-sided bar construction. The monadic two-sided bar construction was introduced in \cite{M} to prove the recognition principle, the theorem that all connected algebras over the little $n$-disks operad are homotopy equivalent to $n$-fold loop spaces. May's proof of the recognition principle immediately generalizes to show that the scanning map for Andrade's model of topological chiral homology is a homotopy equivalence, in the case when the $E_n$-algebra is connected. We will then generalize this proof to the case of non-connected algebras using the Segal spectral sequence for the homology of the geometric realization of a proper simplicial space \cite{Se2}. This will prove Theorem \ref{theorem:main} when we interpret $\int_M A$ to mean Andrade's two sided-bar construction model of topological chiral homology. In Subsection 2.1, we review basic properties of operads and their modules and algebras. In 2.2, we recall the definition of monads and their algebras and functors. In 2.3, we describe the two-sided bar construction and Andrade's model of topological chiral homology. In 2.4, we review properties of simplicial spaces. In 2.5, we review classical theorems about configuration spaces of distinct points in a manifold. In 2.6, we prove nonabelian Poincar\e duality for connected $E_n$-algebras. Finally, in 2.7, we prove that Andrade's model of topological chiral homology exhibits nonabelian Poincar\e duality after stabilizing (Theorem \ref{theorem:main}).

\FloatBarrier
\subsection{Symmetric sequences, operads, modules and algebras}

In this subsection we recall the definition of operads and their modules and algebras. An efficient way of defining operads and their modules is via $\Sigma$-spaces (called symmetric sequences in spoken language). Let $\N$ denote the non-negative integers. 

\begin{definition}
A $\Sigma$-space is a collection of spaces $X(k)$ for all $k \in \N$ such that $X(k)$ has an action of the symmetric group $\Sigma_k$. A map between $\Sigma$-spaces $f:X \m Y$ is a collection of equivariant maps $f_k:X(k) \m Y(k)$.
\end{definition}

The category of $\Sigma$-spaces has a (non-symmetric) monoidal structure defined as follows.

\begin{definition}
For $X$ and $Y$ $\Sigma$-spaces, $X \otimes Y$ is the $\Sigma$-space such that $(X \otimes Y)(k) = \bigsqcup_{j=0}^{\infty} X(j) \times_{\Sigma_j} \bigsqcup_{f \in Map(k,j)} \prod_{i=1}^j Y(|f^{-1}(i)|)$. Here $Map(k,j)$ is the set of maps from $\{1, \ldots k\}$ to  $\{1, \ldots j\}$ and $\Sigma_k$ acts via precomposition.

\end{definition}

Note that the unit with respect to this product is given by the $\Sigma$-space $\iota$ with:  $$\iota(n) = \begin{cases}
 pt, & \text{if }  n=1 \\
\varnothing & \text{if } n \neq 1
\end{cases}$$

\begin{definition}
An operad $\mathcal O$ is a monoid in the category of $\Sigma$-spaces.
\end{definition}

In other words, an operad $\mathcal O$ is a $\Sigma$-space with maps $m : \mathcal O \otimes \mathcal O \m \mathcal O$ and $i: \iota \m \mathcal O$ satisfying the obvious compatibility relations. Denote the image of $\iota$ by $1 \in \mathcal O(1)$.

\begin{definition}
The data of a left module structure on a $\Sigma$-space $\mathcal L$ over an operad $\mathcal O$ is a map $p:\mathcal O \otimes \mathcal L \m \mathcal L $ such that the following diagrams commute:

$$
\begin{array}{ccccccccl}
 \mathcal O \otimes \mathcal O \otimes \mathcal L  &\overset{id \otimes p}{\m} & \mathcal O \otimes \mathcal L & & \iota \otimes \mathcal L  &\overset{i}{\m} & \mathcal O \otimes \mathcal L  \\
m \otimes id \downarrow  & & p \downarrow  &       &    &  \searrow & p \downarrow  \\

\mathcal O \otimes \mathcal L  &\overset{p}{\m} & \mathcal L  & &  & & \mathcal L\\

\end{array}$$

\label{definition:module}

\end{definition}

We likewise define right modules over operads. There is a functor from spaces to $\Sigma$-spaces which sends a space $X$ to the $\Sigma$-space with: $$X(n) = \begin{cases}
 X, & \text{if }  n=0 \\
\varnothing & \text{if } n \neq 0
\end{cases}$$ We will ignore the distinction between a space and its image as a $\Sigma$-space.

\begin{definition}
For an operad $\mathcal{O}$, an $\mathcal{O}$-algebra is a space $A$ with the structure of a left $\mathcal{O}$-module.
\end{definition}

Note that if $X$ is a $\Sigma$-space and $Y$ is a space, then the formula for $X \otimes Y$ simplifies to:
$X \otimes Y=\bigsqcup_k X(k) \times_{\Sigma_k} Y^k$.

\FloatBarrier
\subsection{Monads, right functors and algebras}

From now on, we will also assume that all operads $\mathcal O$ have $\mathcal O(0) =\{0\}$. All algebras that we will consider will have a base point $a_0 \in A$ equal to the image of $\mathcal O(0) \m A$. We will recall the functor on the category of based spaces ($Top_*$) associated to a right module. See \cite{M} for a more detailed treatment of the topics of this section.

\begin{definition}
For $(X,x_0)$ a based space and $\mathcal R$ a right module over an operad $\mathcal O$, define a functor $R: Top_* \m Top_*$ by $ R X = \mathcal R \otimes X / \sim$. Here $\sim$ is the relation that if $r \in \mathcal R$, then $(r;x_1, \ldots x_0 \ldots x_n) \sim (r';x_1, \ldots  x_n)$ with $r'$ the composition of $r$ with $(1,\ldots 1,0, 1 \ldots 1)$.
\end{definition}

We follow the convention of \cite{M} and denote functors associated to operads or right modules by standard font letters. The functor $O$ associated to an operad $\mathcal O$ (viewed as a right module over itself) has more structure than functors coming just from right modules. These types of functors are called monads.

\begin{definition}
A functor $O$, and two natural transformations $\mu : OO \m O$ and $\eta : Id \m O$ are called a monad if the following diagrams commute for every based space $X$:

$$
\begin{array}{cccccccccccl}
OX  &\overset{\eta O}{\m} &OOX & \overset{ O \eta}{ \leftarrow} OX &  & OOOX  &\overset{O \mu }{\m} &OOX  \\
 & id \searrow & \mu \downarrow  &       id \swarrow    & &  \mu O \downarrow   &  & \mu \downarrow  &  \\

  & & OX & & & OOX & \overset{\mu }{\m} & OX   \\

\end{array}$$

\end{definition}

\begin{definition}
Let $O$ be a monad. A space $A$ and a map $\xi :OA \m A$ is called an $O$-algebra if the following diagrams commute:

$$
\begin{array}{ccccccccl}
A  &\overset{\eta }{\m} &OA & & OOA  &\overset{\mu}{\m} &OA\\
 & id \searrow & \xi \downarrow  &  &  O \xi \downarrow   &  & \xi \downarrow  &         \\

  & & A & & OA & \overset{\xi}{\m} & A\\

\end{array}$$

\end{definition}

Note that the data of being an algebra over an operad is the same as the data of being an algebra over the monad associated to that operad.

\begin{definition}
Let $O$ be a monad. An $O$-functor is a functor $R$ and natural transformation $\xi : R O \m R$ making the following diagrams commute for every based space $X$:

$$
\begin{array}{ccccccccl}
RX  &\overset{R \eta }{\m} &ROX & & ROOX &\overset{R \mu}{\m} &ROX  \\
 & id \searrow & \xi \downarrow  &         &  \xi \downarrow   &  & \xi \downarrow   \\

  & & RX & &  ROX & \overset{\xi}{\m} & RX\\

\end{array}$$

\end{definition}

Note that if $\mathcal R$ is a right $\mathcal O$-module, then $R$ is an $O$-functor. Unlike with algebras, there are $O$-functors which do not come from right $\mathcal O$-modules.

\FloatBarrier
\subsection{Topological chiral homology via the two-sided bar construction}

In this subsection, we review a construction of \cite{M} called the two-sided bar construction. May used this construction to prove the recognition principle, the theorem that connected $D_n$-algebras  (see Definition \ref{diskd}) are homotopy equivalent to $n$-fold loop spaces. The two-sided bar construction is a construction which takes as inputs, a monad $O$, an $O$-algebra $A$ and an $O$-functor $R$ and produces a simplicial space $B_*(R,O,A)$. We will then recall Andrade's definition of topological chiral homology which is the two-sided bar construction when $O$ is the monad associated to the little $n$-disks operad and $R$ is a right module defined using embeddings of disks in a manifold.

The space of $k$ simplices of $B_*(R,O,A)$ will be $R O^k A$. The algebra composition map $O A \m A$ induces a map $d_0 : B_k(R,O,A) \m B_{k-1}(M,O,A)$. The monad composition map $O O \m O$ gives $k-1$ maps $d_i : B_k(R,O,A) \m B_{k-1}(R,O,A)$ for $i=1 \ldots k-1$ and the $O$-functor composition map $R O \m R$ gives another map, $d_k : B_k(R,O,A) \m B_{k-1}(R,O,A)$. These maps will be the face map of $B_*(R,O,A)$. The degeneracies, $s_i :  B_k(R,O,A) \m B_{k+1}(R,O,A)$ are induced by the unit of the monad $Id \m O$. In \cite{M}, May noted that these maps satisfy the axioms of face and degeneracy maps of a simplicial space.

\begin{definition}
For $O$ a monad in based spaces, $R$ an $O$-functor and $A$ an $O$-algebra, let $B_*(R,O,A)$ be the simplicial space described above and let $B(R,O,A)$ denote its geometric realization.

\end{definition}

Now we will recall the definition of the little $n$-disks operad $D_n$. We give examples of algebras, modules and functors over $D_n$. These operads, modules and functors will be used to define Andrade's model of topological chiral homology and the scanning map. Let $\D_n$ be the open unit ball in $\R^n$.

\begin{definition}
Let $D_n$ be the $\Sigma$-space with $D_n(k)$ being the space of disjoint axis-preserving affine-linear embeddings of $\bigsqcup_{i=1}^k  \D_n$ into $\D_n$. Topologize this space with the subspace topology inside the space of all continuous maps with the compact open topology. This forms an operad via composition of embeddings.
\label{diskd}
\end{definition}

\begin{figure}
\begin{center}\scalebox{.25}{\includegraphics{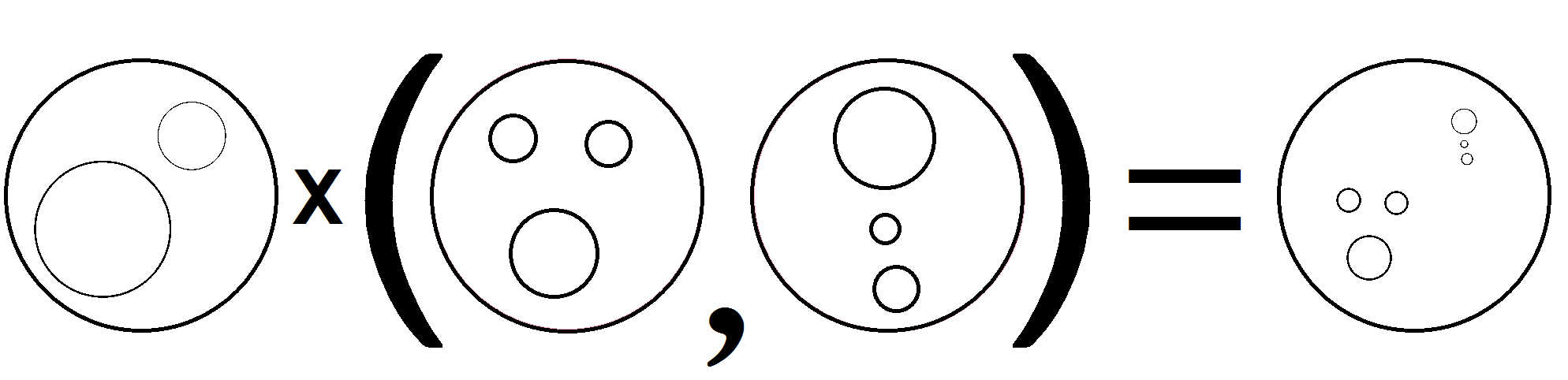}}\end{center}
\caption{Little 2-disks operad}
\label{littedisks}
\end{figure}

For orientable manifolds $N$ and $M$, let $Emb(N,M)$ denote the space of smooth orientation preserving embeddings.

\begin{definition}
Let $M$ be a parallelized $n$-manifold. Using the parallelization, the derivative at a point of an orientation preserving map between two framed $n$-manifolds can be identified with a matrix in $GL_n(\R)^+$ (+ denotes positive determinant). Define $D(M)(k)$ to be the pullback of the following diagram: $$\begin{array}{ccccccccl}
D(M)(k)  & \m & Emb(\bigsqcup_{i=1}^k \D_n,M) \\
 \downarrow & & D_0 \downarrow  &          \\

 Mat_{\R}(n,n)^k & \overset{exp}{\m} & GL^+_n(\R)^k \\

\end{array}.$$ Here the map $D_0: Emb(\bigsqcup_{i=1}^k \D_n,M) \m  GL^+_n(\R)^k$ is the map which records the derivatives at $0 \in \D_n$ of all of the embeddings $\D_n \m M$.
\end{definition}

The spaces $D(M)(k)$ assemble to form a $\Sigma$-space denoted $D(M)$. The purpose of the matrices is to make $D(M)$ homotopy equivalent to the space of configurations of ordered  distinct points in $M$. That is, the map which sends a collection of embeddings to the image of the centers of each disk induces a homotopy equivalence between $D(M)(k)$ and $ M^k-\Delta_{fat}$ where $\Delta_{fat}$ is the fat diagonal. This map is a homotopy equivalence since the subspace of $Emb(\D_n,M)$ consisting of maps with a fixed value and derivative  at $0 \in \D_n$ is contractible (see Chapter 4, Section 5 of \cite{An}). This contrasts with the fact that the fibers of $ Emb(\bigsqcup_{i=1}^k \D_n,M) \m M^k-\Delta_{fat}$ are homotopic to $GL_n^+(\R)^k$.

The space $D(M)$ has the structure of a right $D_n$-module as follows. Ignoring the matrices, the module structure is induced by composition of embeddings. To get the labeling matrices correct, use the following procedure. If $f:\D_n \m M$ is an embedding with matrix $A_0$, and $e:\D_n \m \D_n$ is an affine-linear axis-preserving embedding, pick a path in $D_n(1)$, $e_t$ with $e_0=id$ and $e_1=e$. Next consider the following path of matrices: $B_t=D_{f \circ e_t}(0)$. Since $B_0=exp(A_0)$, the path determines a branch of the logarithm and a unique continuous choice for $A_t$ with $exp(A_t)=B_t$. Associate $A_1$ to the composition $f \circ e$.

To define the scanning map, we need to consider the following $D_n$-functor.

\begin{figure}
\begin{center}\scalebox{.2}{\includegraphics{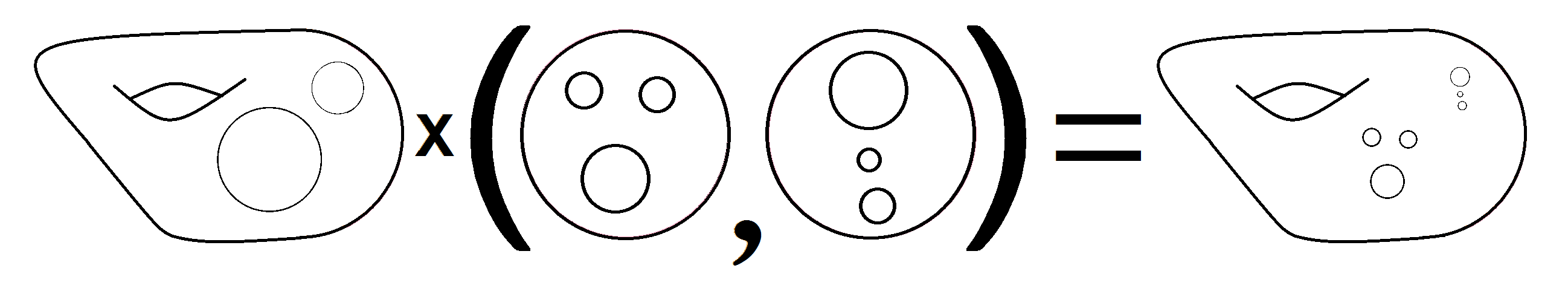}}\end{center}
\caption{Embedding module}
\label{figure:embeddingmodule}
\end{figure}

\begin{example}
Let $Z$ be a paracompact space. The functor $Y \m Map^c(Z , \Simga^n Y)$ is a $D_n$-functor. Here  $Map^c(Z , \Simga^n Y)$ is defined to be the space of based maps from the one point compactification of $Z$ (with $\infty$ as the basepoint) to $Y$.
\end{example}

Note that the standard definition of compactly supported maps is homotopic to this one. We use the convention that the one point compactification of a compact space is the space with a disjoint basepoint. 

This is a $D_n$-functor because the functor $Y \m \Sigma^n Y$ is a $D_n$-functor (page 128 of \cite{M}). For $M$ a parallelized $n$-manifold, there is a scanning natural transformation of $D_n$-functors $s : D(M) \m Map^c(M,\Sigma^n \cdot )$ defined as follows. Let $Y$ be a based space and $e_i: \D_n \m M$ be embeddings. Map the pair $(e_1, \ldots e_k; y_1 \ldots y_k)$ to a map which is constant outside of the images of the $e_i$. For $m \in im(e_i)$, $(e_i^{-1}(m),y_i)$ defines a point inside $\Simga^n Y$. Map $m$ to this point. Here we view $\Sigma^n Y$ as $Y$ smashed with the one point compactification of $\D_n$. This is a natural transformation of $D_n$-functors. 

The scanning natural transformation induces the scanning map $s:B(D(M),D_n,A) \m Map^c(M,B^n A)$ as follows. The natural transformation induces a map $B(D(M),D_n,A) \m B(Map^c(M,\Sigma^n \cdot ),D_n,A)$. Let $X$ be a based space and let $A_*$ be a simplicial space. The spaces of based maps, $Map^{\bullet}(X,A_k)$, assemble to form a simplicial space $Map^{\bullet}(X,A_*)$. There is a natural map $|Map^{\bullet}(X,A_*)| \m Map^{\bullet}(X,|A_*|)$. Taking $X$ equal to the one point compactification of $M$ gives a map $B(Map^c(M,\Sigma^n \cdot),D_n,A) \m  Map^c(M,B( \Sigma^n \cdot ,D_n,A))$. The composition of these two maps is defined to be the scanning map $s:B(D(M),D_n,A) \m Map^c(M,B^n A)$.

\FloatBarrier
\subsection{Properties of simplicial spaces}

In this section, we recall some facts about simplicial spaces that we will need to prove that the scanning map induces a homotopy or homology equivalence. Most of these were proved in \cite{M}. We will also use a lemma appearing in \cite{HM}. Recall that a levelwise weak homotopy equivalence between simplicial spaces does not always induce a weak homotopy equivalence between geometric realizations. The following sufficient condition is due to May in \cite{M}.

\begin{definition}
A simplicial space $A_*$ is called proper if $\cup s_i(A_i) \m A_{i+1}$ is a cofibration for each $i$.
\end{definition}

\begin{theorem}
A map between proper simplicial space $f_* :A_* \m B_*$ induces homology or weak homotopy equivalence on geometric realizations if it does levelwise.
\label{theorem:levelwiseproper}
\end{theorem}

The fact that a levelwise homology equivalence induces a homology equivalence on geometric realizations follows from a spectral sequence introduced by Segal in \cite{Se2}. Given a simplicial space $A_{*}$, let $E^0_{pq}=C_p(A_q)$. This is a double complex with the following two differentials. The first is induced by the differential on singular chains and the second induced by the alternating sum of the face maps. If $A_*$ is proper, this gives a spectral sequence converging to the homology of $|A_{*}|$, the geometric realization of $A_{*}$ with the following $E_2$ page. Let $\partial_{pq}: H_p(A_q) \m H_p(A_{q-1})$ be the alternating sum of the maps in homology induced by the face maps. This forms a chain complex for each $p$. Call this chain complex $\mathcal{E}_p$. The Segal spectral sequence has $E_2^{pq}=H_q(\mathcal{E}_p)$.

To use the Segal spectral sequence or Theorem \ref{theorem:levelwiseproper}, one needs to be able to prove that a given simplicial space is proper. In \cite{M}, May gives the following criterion for $B_*(M,O,A)$ being a proper simplicial space. To state it, May introduces the following definitions.

\begin{definition}
A pair $(X,A)$ of spaces is an NDR-pair if there exists a map $u :X \m [0,1]$ such that $A = u^{-1}(\{0\})$ and a homotopy $h : [0,1] \times X \m X$ with $h(0, x) = x$ for all $x \in X$, $h(t, a) = a$ for all $(t, a) \in [0,1] \times A$, and $h(1, x) \in A$ for all $x \in u^{-1}([0, 1))$. The pair $(h, u)$ is said to be a representation of $(X,A)$ as an NDR-pair.
\end{definition}

\begin{definition}
A functor $F : Top_* \m Top_*$ is admissible if any representation $(h, u)$ of $(X,A)$ as an NDR-pair determines a representation $(Fh, Fu)$ of $(FX, FA)$ as an NDR-pair such that $(Fh)_t = F(h_t)$ on $X$ and such that, for any map $g : X \m X$ with $u(g(x)) < 1$ whenever $u(x) < 1$, the map $Fu : FX \m [0,1]$ satisfies $(Fu)((Fg)(y)) < 1$ whenever $Fu(y) < 1, y \in FX$.
\label{definition:admissable}

\end{definition}

The definition of admissible functor is useful because of the following theorem in \cite{M}.

\begin{proposition}
Let $\mathcal O$ be an operad with $\{id\} \m \mathcal O(1)$ a cofibration of spaces. Let $O$ be the monad associated to $\mathcal O$. Let $A$ be an $\mathcal O$-algebra that is a well based space (inclusion of the base point is a cofibration) and let $M$ be an admissible $O$-functor. Then $B_*(M,O,A)$ is a proper simplicial space.
\label{theorem:admissPropper}
\end{proposition}

May proved that $\Sigma^n$, $\Omega^n$ and monads associated to operads are admissible functors. He noted that composition of admissible functors are also admissible. We shall prove that the functor associated to $D(M)$ and $Map^c(M,\cdot)$ are admissible.

\begin{lemma}
The functor associated to $D(M)$ and $Map^c(M,\cdot)$ are admissible.

\label{lemma:MisAdmissable}
\end{lemma}

\begin{proof}
Let $(h,u)$ represent $(X,A)$ as an NDR-pair. Define $Map^c(M, \cdot )u : Map^c(M, X) \m [0,1]$ by the formula $Map^c(M, \cdot )u (f) = Max_{m \in M} u(f(m))$. Define $D(M)u : D(M)X \m [0,1]$ by $(D(M)u)(e; x_1 \ldots x_k ) = max_i u(x_i)$. These functions satisfy the hypothesis of Definition \ref{definition:admissable}.

\end{proof}

Let $X$ be a based space and let $A_*$ be a simplicial space. The natural map $|Map^{\bullet}(X,A_*)| \m Map^{\bullet}(X,|A_*|)$ is not always a homotopy equivalence. In \cite{M}, May proved that if each $A_k$ is connected, $A_*$ is proper, and $X$ is a circle, then the map is a weak equivalence. In \cite{HM}, they proved that the map is a homotopy equivalence if $X$ is a simplicial complex, $A_*$ is proper and $A_k$ is $dim X$ connected for each $k$. We will generalize the last theorem by relaxing the connectivity of requirements on the spaces $A_k$ by one. These proofs use the notion of a simplicial Hurewicz fibration introduced in \cite{M}. A simplicial Hurewicz fibration is a condition on a map of simplicial spaces $f_*: E_* \m B_*$ that generalizes the notion of fibration of spaces. This definition is somewhat involved so we will not give it. However, we will state two theorems. In \cite{HM} (page 11), Hesselholt and Madsen observed that a particular class of maps are simplicial Hurewicz fibrations. In \cite{M} (Theorem 12.7), May proved useful properties of simplicial Hurewicz fibrations.

\begin{proposition}
Let $Z$ and $X$ be simplicial complexes and let $X = Z \cup e_l$ with $e_l$ an $l$-cell. Then for any proper simplicial space $A_*$, $Map^{\bullet}(X,A_*) \m Map^{\bullet}(Z,A_*)$ is a simplicial Hurewicz fibration with fiber $Map^{\bullet}(S^l,A_*)$.

\end{proposition}

In \cite{HM}, they also noted that $Map^{\bullet}(X,A_*)$ is proper if $A_*$ is.  

\begin{proposition}
Let $E_* \m B_*$ be a simplicial Hurewicz fibration with fiber $F_*$. If each $B_k$ is connected and $B_*$ is proper, then $|E_*| \m |B_*|$ is a quasi-fibration with fiber $|F_*|$.

\end{proposition}

Combining these two propositions, we get the following corollary.

\begin{corollary}

Let $X$ be a finite simplicial complex of dimension at most $n$. Assume that $A_*$ is a proper simplicial space and that each $A_k$ is $(n-1)$-connected. Then the map $ |Map^{\bullet}(X,A_*)| \m Map^{\bullet}(X,|A_*|) $ is a weak equivalence.
\label{corollary:mapcommute}
\end{corollary}


\begin{proof} Let $X_{n-1}$ be the $n-1$-skeleton of $X$. The quotient $X/X_{n-1}$ is a wedge of spheres $\bigvee S^n$. For any space $Y$, $Map^{\bullet}(\bigvee S^n,Y) \m Map^{\bullet}(X,Y) \m Map^{\bullet}(X_{n-1},Y)$ is a fibration sequence since $X_{n-1} \m X$ is a cofibration. This is true in particular when $Y=|A_*|$. Since every $A_k$ is $n-1$ connected and $X_{n-1}$ is $n-1$ dimensional, the spaces $Map^{\bullet}(X_{n-1},A_k)$ are connected. The sequence $|Map^{\bullet}(\bigvee S^n,A_*)| \m |Map^{\bullet}(X,A_*)| \m |Map^{\bullet}(X_{n-1},A_*)|$ is a quasi-fibration sequence since $Map^{\bullet}(\bigvee S^n,A_*) \m Map^{\bullet}(X,A_*) \m Map^{\bullet}(X_{n-1},A_*)$ is a simplicial Hurewicz fibration, $A_*$ is proper and $Map^{\bullet}(X_{n-1},A_k)$ is connected for every $k$. Consider the following commuting diagram of quasi-fibrations: $$
\begin{array}{ccccccccl}
|Map^{\bullet}(\bigvee S^n,A_*)| & \m  & Map^{\bullet}(\bigvee S^n,|A_*|)               \\
 \downarrow  &    & \downarrow             \\

|Map^{\bullet}(X,A_*)| & \m  & Map^{\bullet}(X,|A_*|)  \\

 \downarrow  &    & \downarrow             \\

|Map^{\bullet}(X_{n-1},A_*)| & \m  & Map^{\bullet}(X_{n-1},|A_*|) .

\end{array}$$ The map $|Map^{\bullet}(\bigvee S^n,A_*)| \m Map^{\bullet}(\bigvee S^n,|A_*|)$ is a weak equivalence by \cite{M} (Theorem 12.3) and $|Map^{\bullet}(X_{n-1},A_*)| \m Map^{\bullet}(X_{n-1},|A_*|)$ is a weak equivalence by \cite{HM} (Lemma 1.4). Thus,  $|Map^{\bullet}(X,A_*)| \m Map^{\bullet}(X,|A_*|)$ is a weak equivalence.
\end{proof}

\FloatBarrier
\subsection{Classical scanning theorems for configuration spaces}

To prove the recognition principle, May proved the so called approximation theorem. Namely he proved that $s : D_n X \m \Omega^n \Sigma^n X$ is a weak equivalence for $X$ connected and well based \cite{M}. In order to prove nonabeilian Poincar\e duality theorems, we will need similar results concerning the scanning map $s : D(M) X \m Map(M,\Sigma^n X)$. Thus, we review some facts about classical configuration spaces proved in \cite{Mc1} and \cite{B}.

\begin{definition}
Let $M$ be a parallelized $n$-manifold. Let $C(M)$ be the $\Sigma$-space with $C(M)(k)= M^k - \Delta_{fat}$ with $\Delta_{fat}$ denoting the fat diagonal.   
\end{definition}

Note that the map $c:D(M)(k) \m C(M)(k)$ sending an embedding to its center is a homotopy equivalence and so $c$ induces a homotopy equivalence between $D(M)X \m C(M)X$ for all well based spaces $X$. In \cite{B}, B{\"o}digheimer defined scanning maps $s'$ making the following diagram homotopy commute: $$
\begin{array}{ccccccccl}
D(M)X  &\overset{c}{\m} & C(M)X \\
 & s \searrow & s' \downarrow  &            \\

  & & Map^c(M,\Sigma^n X) \\

\end{array}.$$ B{\"o}digheimer also generalizes May's approximation theorem and proves the following theorem \cite{B}. 

\begin{theorem}
If $X$ is connected, $M$ a parallelizable $n$-manifold, then $s': C(M)X \m Map^c(M,\Sigma^n X)$ is a homotopy equivalence.
\label{theorem:connectedConfig}

\end{theorem}

If $X$ is not connected, then $s$ and $s'$ are not homotopy equivalences. However, if $M$ is open, they are what we shall call a ``stable'' homology equivalence. The word stable does not mean stable in the sense of stable homotopy theory. For simplicity we assume that $M$ is connected and is the interior of a (not necessarily compact) manifold $\bar M$ with connected boundary $\partial M$. However, everything can be generalized to the case when $\pi_0(\partial M) \m \pi_0(\bar M)$ is onto.


Let $M' = \bar M \cup_{\partial M} \partial M \times [0,1)$. Fix a diffeomorphism $d: M' \m M$ whose inverse is isotopic to the inclusion $M \hookrightarrow M'$. Given $x \in X$ and $p \in \partial M \times [0,1)$, there is an induced map $t_x: C(M)X \m C(M)X$ defined as follows. Send a configuration $(m_1, \ldots m_k; x_1 \ldots x_k)$ to $(d(m_1), \ldots d(m_k),d(p); x_1 \ldots x_k,x)$. Up to homotopy, $t$ only depends on $[x] \in \pi_0(X)$. Let $f_x:\partial M \times [0,1) \m \Sigma^n X$ be $s(p;x)$. Let $T_x :Map^c(M,\Sigma^n X) \m Map^c(M,\Simga^n X)$ be the following function:

$$T_x(f)(m)=
\begin{cases}
f(d^{-1}(m))  & \mbox{if } d^{-1}(m) \in M \\
f_x(d^{-1}(m))  & \mbox{if } d^{-1}(m) \notin M . \\
\end{cases}
$$

Let $\{x_i\}$ be a sequence of not necessarily distinct elements of $X$ such that each connected component of $X$ has infinitely many terms of the sequence. The natural numbers $\N$ are a partially ordered set and hence a category. Let $\mathfrak C: \N \m Top$ be a functor that takes each object to $C(M)X$ and sends morphisms $j \m j+1$ to maps $t_{x_j}$. Likewise define $\mathfrak M:\N \m Top$ to be the functor which takes objects to $Map^c(M,\Sigma^n X)$ and morphisms to the maps $T_{x_{j}}$. In \cite{Mc1}, McDuff proved the following theorem in the case that $X=S^0$.

\begin{theorem}
If $M $ is a connected parallelizable manifold, $\partial M$ is non-empty and $dim_{\R} M>1$, then $s':hocolim_{\N} \mathfrak C \m hocolim_{\N} \mathfrak M$ is a homology equivalence.

\label{theorem:limitConfig}
\end{theorem}

However, all of her arguments follow almost verbatim for general well based $X$. See \cite{B} for more details. We shall describe the above theorem by saying that $s: C(M)X \m Map^c (M,\Sigma^n X)$ is a stable homology equivalence. We call the maps $t_x$ and $T_x$ stabilization maps. Note that each $T_x$ are homotopy equivalences. Thus $hocolim_{\N} \mathfrak M$ is homotopy equivalent to $Map^c(M, \Sigma^n X)$.

\FloatBarrier
\subsection{Scanning with a connected algebra}

The goal of this section is to prove that the scanning map $s:B(D(M), D_n, A) \m Map^c(M,B^n A)$ is a weak homotopy equivalence when $A$ is connected. The proof follows May's proof of the approximation theorem in \cite{M}. This is a special case of the Lurie's nonabelian Poincar\e duality theorem from \cite{Lu} since all connected $D_n$-algebras are grouplike. Despite being less general than theorems already appearing in the literature, we give this proof as warm up to the case when $A$ is not necessarily grouplike. 

\begin{theorem}
If $A$ is a connected $D_n$-algebra, $M$ a parallelized $n$-manifold, then the scanning map $s:B(D(M), D_n, A) \m Map^c(M,B^n A)$ is a weak homotopy equivalence.
\end{theorem}

\begin{proof} 
Consider the map of simplicial spaces: $$B_*(D(M), D_n, A) \m B_*(Map^c(M,\Sigma^n \cdot), D_n, A).$$ By Lemma \ref{lemma:MisAdmissable} and Proposition \ref{theorem:admissPropper}, both simplicial spaces are proper. By Theorem \ref{theorem:connectedConfig} and the fact that $D(M)X$ is homotopic to $C(M)X$, the map is a levelwise homotopy equivalence. Thus by Theorem \ref{theorem:levelwiseproper}, it induces a weak equivalence on geometric realizations. Now consider the map: $$B(Map^c(M,\Sigma^n \cdot), D_n, A) \m Map^c(M, B(\Sigma^n , D_n, A)).$$ Since $B_*(\Sigma^n , D_n, A)$ is a proper simplicial space and $\Sigma^n D^k_n A$ is $n$-connected (and hence $n-1$ connected) for every $k$, by Corollary \ref{corollary:mapcommute} we can conclude that this map is a weak homotopy equivalence. Since the scanning map is the composition of these two maps, it too is a weak  homotopy equivalence.
\end{proof}

\FloatBarrier
\subsection{Scanning with an open manifold}

In this subsection, we no longer assume that the $D_n$-algebra $A$ is connected or even grouplike. As before, we assume that $M$ is a smooth, connected, open, parallelizable $n$-manifold with $n>1$. We also assume that $M$ is the interior of a (not necessarily compact) manifold $\bar M$ with connected boundary. Note that insisting that the boundary of $\bar M$ is connected is not a condition on $M$ since we do not assume that $\bar M$ is compact. The goal of this section is to prove that the scanning map $s:B(D(M), D_n, A) \m Map^c(M,B^n A)$ is a stable homology equivalence. Here stable is in the sense of Subsection 2.5. Intuitively, we want to use the same proof strategy as was used in the case where $A$ is connected except using Theorem \ref{theorem:limitConfig} instead of Theorem \ref{theorem:connectedConfig}. See Subsection 2.5 for definitions of the stabilization maps $t_x: C(M)X \m C(M)X$ and $T_x : Map^c(M; \Sigma^n X) \m Map^c(M;\Sigma^n X)$.

For $\alpha \in B_k(D(\partial M \times (0,1)), D_n,A)$, we can define a stabilization map $t_{\alpha}:  B_k(D(M), D_n,A) \m  B_k(D(M), D_n,A)$ in a similar way to the way $t_x: C(M)X \m C(M)X$ was defined. That is, taking the ``union'' of an element $\xi \in B_k(D(M),D_n,A)$ with the element $\alpha \in B_k(D(\partial M \times (0,1)), D_n,A)$ gives an element of $B_k(D(\bar M \cup_ {\partial M} \partial M \times [0,1)), D_n,A)$. Using a diffeomorphism $\bar M \cup_ {\partial M} \partial M \times (0,1) \m M$ produces an element of  $B_k(D(M), D_n,A)$. For future use, require that the inverse of this diffeomorphism is isotopic to the inclusion $M \hookrightarrow \bar M \cup_ {\partial M} \partial M \times [0,1)$.

For $a \in A$, we shall define a stabilization map $t_a : B_*(D(M),D_n,A) \m B_*(D(M),D_n,A)$ as follows. Let $a^0 \in D(\partial M \times (0,1))A$ be a disk labeled by $a$. Let $a^k \in B_k(D(\partial M \times (0,1)), D_n, A)$ be the image of $a^0$ under the composition of $k$ degeneracy maps. Let $t_{a}:B_*(D(M),D_n,A) \m B_*(D(M),D_n,A)$ be the map induced by $t_{a^k} : B_k(D(M),D_n,A) \m B_k(D(M),D_n,A)$. 
Let $\{a_i\}$ be a sequence of not necessarily distinct elements of $A$ such that each connected component of $A$ has infinitely many terms of the sequence. Let $\mathfrak C: \N \m Top$ be a functor that takes each object to $B(D(M),D_n,A)$ and sends morphisms $j \m j+1$ to maps $t_{a_j}$. Likewise define $\mathfrak M:\N \m Top$ to be the functor which takes objects to $Map^c(M,B^n X)$ and morphisms to the maps $T_{a_j}$ (defined in the analogous way using the maps $T_{a^k_j}:B_k(Map^c(M,\Sigma^n \cdot),D_n,A) \m B_k(Map^c(M,\Sigma^n \cdot),D_n,A))$. 

Also define $\mathfrak C_k$ and $\mathfrak M_k$ to be the analogous functor sending every natural number to $ B_k(D(M),D_n,A)$ and $B_k(Map^c(M,\Sigma^n \cdot),D_n,A)$ respectively. Let $C_*$ be the simplicial space with $C_k =hocolim_{\N} \mathfrak C_k$ and face maps and degeneracies induced by the face maps and degeneracies of $B_*(D(M),D_n,A)$. Likewise define $M_*$ to be the simplicial space with $M_k =hocolim_{\N} \mathfrak M_k$. We will show that $|C_*|$ is homology equivalent to $|M_*|$. 

We will need to compare the maps  $t_{\alpha}$ with the maps  $t_{a^k}$. Let $\partial_k: H_*(B_k(D(M), D_n,A)) \m H_*(B_{k-1}(D(M), D_n,A))$ be the alternating sum of the facemaps.

\begin{lemma}
For all $x \in ker( \partial_k)$ and $\alpha \in B_k(D(\partial M \times (0,1)), D_n,A)$, there exist $a \in A$ such that $t_{\alpha*}(x) -t_{a^k*}(x) \in im(\partial_{k+1})$. In other words, $t_{\alpha*}$ and $t_{a^k*}$ induce the same map on the $E_2$ page of the Segal spectral sequence for the homologies of $|C_*|$ and $|M_*|$.
\label{lemma:stabequiv}
\end{lemma}

\begin{proof}
For $\alpha, \alpha' \in B_k(D(\partial M \times (0,1)), D_n,A)$, we say  $t_{\alpha*}$ and $t_{\alpha'*}$  are homologous to mean   $t_{\alpha*}(x) -t_{\alpha'*}(x) \in im(\partial_{k+1})$ for all $x \in ker( \partial_k)$. An element $\alpha \in B_k(D(\partial M \times (0,1)), D_n,A)$ includes the data of a collection of embeddings $e \in D(\partial M \times (0,1))$ and elements $e_{ij} \in D_n$ for $0 \leq i \leq k $ and elements $a_l \in A$. We order the $e_{ij}$ so that the elements $(e_{k1}, e_{k2} \ldots ; a_1, a_2 \ldots)$ define elements of $D_n A$ and $(e_{k-1,1}, e_{k-1 ,2} \ldots ;  e_{k1}, e_{k2} \ldots ; a_1, a_2 \ldots)$ define elements of  $D_n(D_n A)$, et cetera. For all $N>0$, we say that $\alpha \in B_k(D(\partial M \times (0,1)), D_n,A)$ has at least $N$ outer walls if $e \in D(\partial M \times (0,1) )(1)$ and each $e_{ij} \in D_n(1)$ for $i<N$ (See Figure \ref{Walls}). We say that all elements have at least zero outer walls. For $\alpha, \alpha' \in B_k(D(\partial M \times (0,1)), D_n,A)$, we say that $\alpha \sim \alpha'$ if $\alpha$ and $\alpha'$ are in the same connected component of $B(D(\partial M \times (0,1)), D_n,A)$. If $\alpha \sim \alpha'$ and $\alpha$ and $\alpha'$ both have $k+1$ outer walls, then $\alpha$ and $\alpha'$ are in the same component of $B_k(D(\partial M \times (0,1)), D_n,A)$ and hence $t_{\alpha*} =t_{\alpha'*}$. Note that $t_{\alpha}$ is homotopic to $t_{a^k}$ for some $a$ if $\alpha$ has $k+1$ outer walls.

\begin{figure}
\begin{center}
\scalebox{.5}{\includegraphics{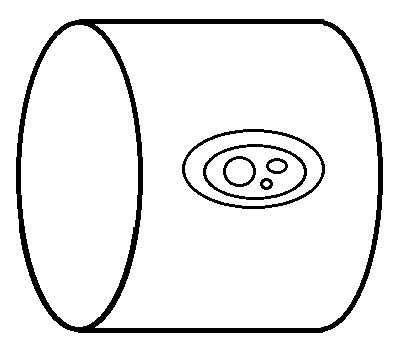}}
\caption{Two outer walls}
\label{Walls}
\end{center}
\end{figure}

We shall assume for the purposes of induction that if $\alpha \sim \alpha'$ and $\alpha$ and $\alpha'$ both have $N+1$ outer walls, then $t_{\alpha*}$ and $t_{\alpha'*}$ are homologous. We shall now prove that if $\alpha \sim \alpha'$ and $\alpha$ and $\alpha'$ both have $N$ outer walls, then $t_{\alpha*}$ and $t_{\alpha'*}$ are homologous. Let $\beta \in B_{k+1}(D(\partial M \times (0,1)), D_n,A)$ be $\alpha$ with an extra outer wall. Note that $d_{i} \beta =\alpha$ for $i \geq k+1 -N$ and $d_{i} \beta$ has $N+1$ outer walls for $i<k+1 -N$ (see Figure \ref{AlphaBetaAlphaPrime} for an illustration). Let $a \in A$ be the product of all of the elements of $A$ labeling the disks comprising $\alpha$. This is well defined in $\pi_0(A)$. Note that $a^k$ has $k+1$ (and hence $N+1$) outer walls (see above for definition of $a^k$). Also observe that $a^k \sim d_i \beta$ for all $i$. By our induction hypothesis, $t_{a^k*}$ is homologous to  $t_{d_i \beta*}$ for $i<k+1 -N$ since $d_i \beta$ has $N+1$ outer walls for $i<k+1 -N$. Let $x \in ker \partial_k$ be arbitrary. Note that the following diagram commutes:

\begin{equation}
\begin{array}{ccccccccl}
B_{k+1}(D( M), D_n,A)  & \overset{t_{\beta}}{\m}  & B_{k+1}(D( M), D_n,A)                 \\
 d_i \downarrow  &    & d_i \downarrow             \\

B_{k}(D( M), D_n,A)  & \overset{t_{d_i \beta}}{\m}  & B_{k}(D( M), D_n,A)

\end{array}
\label{alphabeta}
\end{equation}

 Let $s_l: B_{k}(D(M), D_n,A) \m B_{k+1}(D(M), D_n,A)$ be a degeneracy map. Let $x \in ker \partial_k$ be arbitrary.  From now on, we will write ``$=$'' for homologous and not distinguish between maps and their effects in homology.  Note that for any $z \in H_*(B_{k-1}(D(M), D_n,A) )$, $s_l z$ is null homologous. We have: $$0=\partial_{k+1} t_{\beta} s_l x \text{ (since boundaries are zero in homology)} $$ $$= \sum_{i=0}^{i=k+1} \pm d_i t_{\beta} s_l x \text{ (by the definition of } \partial)$$ $$=\sum_{i=0}^{i=k-N} \pm t_{a^k} d_i s_l x+  \sum_{i=k+1-N}^{i=k+1} \pm t_{\alpha} d_i s_l x \text{ (by Diagram (\ref{alphabeta})) }$$

$$\text{ and}$$

$$0=\partial_{k+1} t_{a^{k+1}} s_l x \text{ (since boundaries are zero in homology)}$$ $$=t_{a^k} \partial_{k+1} s_l x \text{ (since } t_a \text{ is a chain map)}$$ $$ =\sum_{i=0}^{i=k-N} \pm t_{a^k} d_i s_l x+  \sum_{i=k+1-N}^{i=k+1} \pm t_{a^k} d_i s_l x \text{ (definition of } \partial).$$ Subtracting, we see that $t_{\alpha}(y) = t_{a^k}(y)$ for $y= \sum_{i=k+1-N}^{i=k+1} \pm d_i s_l x$. Let $l =k-N$. We have:  $$ y =\pm d_{k+1-N} s_{k-N} x + \sum_{i=k+2-N}^{i=k+1} \pm d_i s_{k-N} x$$ $$=\pm x + \sum_{i=k+2-N}^{i=k+1} \pm  s_{k-N} d_{i-1} x \text{ (simplicial identities)}$$ $$= \pm x  \text{ (image of degeneracies are null homologous.) }$$ Thus $x= \pm y$ so $t_{\alpha}(x) =t_{a^k}(x)$. Since $x$ was arbitrary, we can conclude that $t_{\alpha}$ and $t_{a^k}$ are homologous. If $\alpha' \sim \alpha$, then $\alpha' \sim a^k$ and so $t_{\alpha}$ and $t_{\alpha'}$ are also homologous. The claim now follows by induction.

\end{proof}

\begin{theorem}
The scanning map $s:B(D(M),D_n,A) \m Map^c(M,B^n A)$ induces a map of simplicial spaces $s_* : C_* \m M_*$ which induces a homology equivalence on geometric realizations.
\label{theorem:main2}
\end{theorem}

\begin{proof}
First note that $C_*$ and $M_*$ are proper simplicial spaces. This follows from the fact that $B_*(D(M),D_n,A)$ and $B_*(Map(M,\Sigma^n \cdot),D_n,A)$ are proper and the following fact about cofibrations. Assume that the following diagram commutes and the vertical maps are cofibrations.  $$
\begin{array}{ccccccccl}
 X  & \m  & Y               \\
 \downarrow  &    & \downarrow             \\

Z  &\m & W

\end{array}$$ Then the inclusion of the mapping cylinder of $X \m Y$ into the mapping cylinder of $Z \m W$ is a cofibration.
 
It is not true that the scanning maps $s:C_k \m M_k$ are homology equivalences. This would follow from Theorem \ref{theorem:limitConfig} if we took the homotopy colimit with respect to stabilization maps $t_{\alpha}$ for $\alpha$ representing each component of $\pi_0(B_k(D(\partial M \times (0,1)),D_n,A))$. However, we are only using stabilization maps of the form $t_{a^k_i}$ (see above for notation). Fortunately this difference is not relevant on the $E_2$ page of the Segal spectral sequence (see Subsection 2.4).  By Lemma \ref{lemma:stabequiv}, the difference between the effects in homology of $t_{\alpha}$ with $\alpha$ arbitrary and $t_{a^k_i}$ is in the image of the alternating sum of the face maps $\sum \pm d_i : H_*(C_{k+1}) \m H_*(C_k)$. Thus by Theorem \ref{theorem:limitConfig}, the map, $s : C_* \m M_*$ induces an isomorphism on the $E_2$ page of the Segal spectral sequence and hence $s : |C_*| \m |M_*|$ is a homology equivalence.
\end{proof}

We can now deduce Theorem \ref{theorem:main} (nonabelian Poincar\e duality after stabilizing) for Andrade's model of topological chiral homology, $\int_M A =B(D(M),D_n,A)$.

\begin{corollary}
If $M$ is an open parallelizable $n$-manifold, the scanning map $s$ induces a homology equivalence between $hocolim_{t_{a_i}} B(D(M),D_n,A)$ and $Map^c(M,B(\Sigma^n,D_n,A))$.
\end{corollary} 

\begin{proof}
By Theorem \ref{theorem:main2}, $s:|C_*| \m |M_*|$ is a homology equivalence. After interchanging colimits and using Corollary \ref{corollary:mapcommute}, we have that $|C_*| = hocolim_{t_{a_i}} B(D(M),D_n,A)$ and $|M_*|=hocolim_{T_{a_i}} Map(M,B(\Sigma^n,D_n,A)$. Since the stabilization maps $T_{a_i}: Map^c(M,B(\Sigma^n,D_n,A)) \m Map^c(M,B(\Sigma^n,D_n,A))$ are homotopy equivalences, $Map^c(M,B(\Sigma^n,D_n,A))$ is weakly equivalent to $|M_*|$. Thus $hocolim_{t_{a_i}} B(D(M), D_n, A)$ is homology equivalent to $ Map^c(M,B(\Sigma^n,D_n,A))$ and we have proven Theorem \ref{theorem:main} for Andrade's model of topological chiral homology.
\end{proof}

\begin{figure}
\begin{center}
\scalebox{.25}{\includegraphics{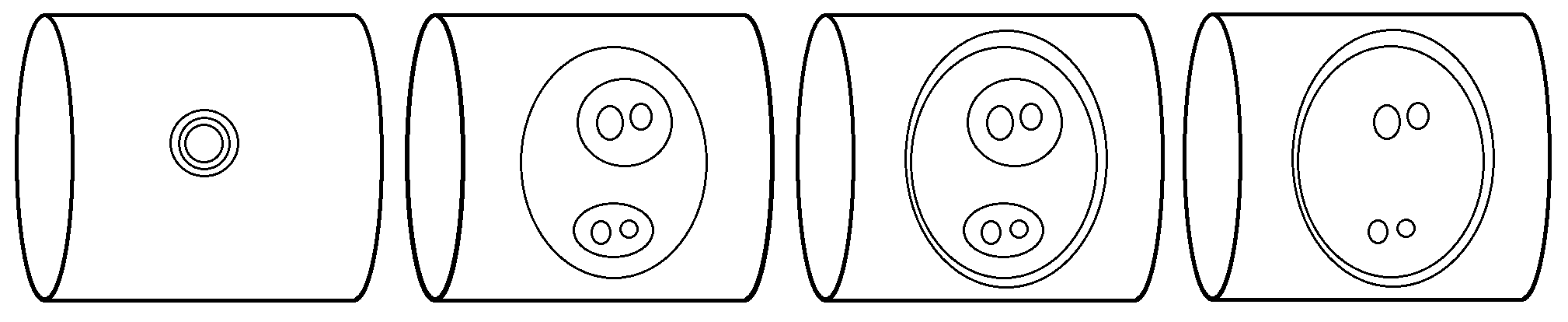}}
\caption{Example of $a^k$, $\alpha$, $\beta$, and $d_i \beta$ with $i<k+1 -N$}
\label{AlphaBetaAlphaPrime}
\end{center}
\end{figure}

\FloatBarrier
\section{Configuration spaces of particles with summable labels}

The goal of this section is to prove Theorem \ref{theorem:main} (nonabelian Poincar\e duality after stabilizing) for Salvatore's model of topological chiral homology, configuration spaces of particles with summable labels \cite{Sa}. We will no longer restrict our attention to the case of \p manifolds. Thus we will consider framed $E_n$-algebras instead of unframed. 

Unlike other models of topological chiral homology, Salvatore's model does not take as inputs an algebra over the little $n$-disks operad, but instead accepts algebras over the (framed) Fulton-MacPherson operad $F_n$. In the previous section, the arguments can be summarized as using May's two-sided bar construction to leverage classical results about configuration spaces to draw conclusions about topological chiral homology. In contrast, Salvatore's model of topological chiral homology is sufficiently close to classical configuration spaces that many of those arguments developed for classical configuration spaces directly apply.  In Subsection 3.1, we recall the definition of the (framed) Fulton-MacPherson configuration spaces and operad as well as review the definition of configuration spaces of particles with summable labels. In 3.2, we recall Salvatore's definition of relative configuration spaces. In 3.3, we extend Salvatore's nonabelian Poincar\e duality theorem from \cite{Sa} to more general relative configuration spaces using arguments from \cite{B}. In 3.4, we use ideas similar to those used by McDuff in \cite{Mc1} regarding homology fibrations and configuration spaces to prove non-abelian Poincar\e duality after stabilizing for Salvatore's model of topological chiral homology.

\FloatBarrier
\subsection{Fulton-Macpherson operad and configuration spaces}

In this subsection, we will recall the definition of the Fulton-MacPherson configuration space and the Fulton-MacPherson operad. Using these definitions, we will describe Salvatore's model of topological chiral homology. We follow the treatment in \cite{Sa} in general but modify some notation in order not to conflict with notation from the previous section. The Fulton-MacPherson configuration space is a partial compactification of the configuration space of ordered distinct points in a manifold.

\begin{definition}
Let $Bl_{\Delta} M^k$ denote the real oriented blow-up of $M^k$ along the small diagonal $\Delta$. 
\end{definition}

There is a natural map from $C(M)(k)\m Bl_{\Delta}(M^k)$. Given a subset $S \subset \{1, \ldots k\}$, there is a natural map $C(M)(k) \m C(M)({|S|})$. Let $j : C(M)(k) \m \prod_{|S|>1} Bl_{\Delta} M^S$ be the product of these maps.

\begin{definition}
Let $C^{fm}(M)(k)$ be the closure of the image of $j$.
\end{definition}

There is an obvious macroscopic location map $b: C^{fm}(M)(k) \m M^k$.

\begin{definition}
As a space, the $k$'th space of operad $F_n$ is the collection of points of $C^{fm}(\R^n)(k)$ macroscopically located at the origin, $b^{-1}(0,0, \ldots)$.
\end{definition}

See \cite{GJ} or \cite{Sa} for an operad structure on $F_n$ and a proof that $F_n$ is an $E_n$-operad. To define the framed \fm operad, we first recall the definition of the semi-direct product of a group $G$ and an operad $\mathcal{O}$ in $G$-spaces (an operad in spaces where each space has a $G$-action and all structure maps are $G$-equivariant).

\begin{definition}
Let $G$ be a group and let $\mathcal{O}$ be an operad in $G$-spaces. Define $\mathcal O \rtimes G$ to be the topological operad with $(\mathcal O \rtimes G)(k)=\mathcal O(k) \times G^k$ and composition map $\tilde m: (\mathcal O \rtimes G) \otimes ( \mathcal O \rtimes G) \m \mathcal O \rtimes G$ given by: $$\tilde m((o,g_1, \ldots g_k);(o_1,g^1_1, \ldots g^{m_1}_1), \ldots,(o_k,g^1_k, \ldots g^{m_k}_k)) =$$ $$( m(o;g_1 o_1, \ldots, g_k o_k),g_1g_1^1,\ldots, g_kg_k^{m_k})$$ with $m:\mathcal O \otimes \mathcal O \m \mathcal O$ the composition map of the operad $\mathcal O$. 
\end{definition}

We now define the framed \fm operad and the framed \fm configuration space.

\begin{definition}
Define the framed \fm operad $fF_n$ as $F_n \rtimes GL_n(\R)$ with the $GL_n(\R)$ action on $F_n(k)$ induced by the action of $GL_n(\R)$ on $\R^n$. 
\end{definition}

See \cite{Sa} for more details on this action.


\begin{definition}
Let $P \m M$ be the principle $GL_n(\R)$ bundle associated to the tangent bundle of an $n$-manifold $M$. Let the framed \fm configuration space of $k$ points in $M$ be the pullback of the following diagram:$$
\begin{array}{ccccccccl}
fC^{fm}(M)(k) &\m & P^k\\
 \downarrow   & & \downarrow       \\
C^{fm}(M)(k) & \overset{b}{\m} & M^k\\

\end{array}$$ with $b$ the macroscopic location map. Let $fC^{fm}(M) = \bigsqcup_k fC^{fm}(M)(k)$. Note that the symmetric group actions on $M^k$ induces a  $\Sigma$-space structure on $fC^{fm}(M)$.
\end{definition}

In \cite{Sa}, Salvatore describes a right $fF_n$-module structure on $fC^{fm}(M)$. This structure is similar to the $D_n$-module structure on $D(M)$. For $M$ an $n$-manifold with corners, we define $fC^{fm}(M)$ as follows. Let $W$ be an $n$-manifold (without boundary or corners) containing $M$. Let $fC^{fm}(M)$ be the subspace of $fC^{fm}(W)$ of particles macroscopically located in $M$. We can now give Salvatore's definition of the configuration space of particles with summable labels.

\begin{definition}
For an $fF_n$-algebra $A$, $M$ an  $n$-manifold (possibly with corners), let $C(M;A)$ be the coequalizer of the two natural maps $fC^{fm}(M) \otimes fF_n \otimes A \rightrightarrows fC^{fm}(M) \otimes A$.
\end{definition}

\FloatBarrier
\subsection{Relative configuration spaces of particles}

There also is a relative version of Salvatore's configuration spaces of particles with summable labels. For $N \subset M$, we will define a relative configuration space $C^{fm}(M,N;A)$. Intuitively it is the space of particles in $M$ which vanish if they enter $N$. To define $C(M;A)$, we used a right $fF_n$-module $fC^{fm}(M)$. To define $C^{fm}(M,N;A)$, we will need to define a right functor $fC(M,N)$ over the monad $fF_n$. Throughout, we will assume that $M$ and $N$ are manifolds with corners, $N \hookrightarrow M$ is a cofibration, $dim M=n$ and that $M - N$ is open. 

Let $W$ be an open submanifold containing $N$. Let $X$ be a space. All elements of $fC^{fm}(M)X$ can be uniquely described by an element of $fC^{fm}(M-N)X$ and an element of $fC^{fm}(W)X$ consisting of points macroscopically located in $N$. Let $\sim$ be the relation that on $fC^{fm}(M)X$ given by identifying elements whose corresponding elements in $fC^{fm}(M-N)X$ are equal. 

\begin{definition}
For a based space $X$, we define a space $fC^{fm}(M,N)X$ to be $C^{fm}(M)X / \sim$. Let $fC^{fm}(M,N)$ be the functor from based spaces to based spaces which sends a space $X$ to $fC^{fm}(M,N)X$.
\end{definition}

See \cite{Sa} for a description of the right $fF_n$-functor structure on $C^{fm}(M,N)$.

\begin{definition}
For $A$ an $fF_n$-algebra, let $C(M,N;A)$ denote the coequalizer of the two natural maps $fC^{fm}(M,N)  fF_n A \rightrightarrows  fC^{fm}(M,N) A$.
\end{definition}
\FloatBarrier
\subsection{Quasifibrations and scanning theorems}

In this subsection we will recall Salvatore's definition of the scanning map, review the nonabelian Poincar\e duality theorems from \cite{Sa}, as well as review Salvatore's results concerning when the natural map $\pi: C(M;A) \m C(M,N;A)$ is a quasi-fibration. All theorems without proof in this subsection are due to Salvatore in \cite{Sa}. Before we can define Salvatore's scanning map, we need the following theorem regarding a model of $B^nA$ and the following definition of relative compactly supported space of sections.

\begin{theorem} For an $fF_n$-algebra $A$, $C(S^n, pt; A) \simeq B^nA$ and $C(\R^n;A) \simeq A$.
\end{theorem}

For a manifold $M$ with a metric, let $TM$, $DM$, and $SM$ respectively denote the tangent bundle, closed unit disk bundle and sphere bundle of $M$. Given a $fF_n$-algebra $A$, let $E^A \m M$ be the bundle whose fiber over a point $m \in M$ is $C(DM_m,SM_m;A)$. We will define a scanning map from $C(M;A)$ to a space of sections of $E^A$. Note that fibers of $E^A$ are all models of $B^nA$.

\begin{definition}
Let $W$ be an open $n$-manifold containing $M$ containing $N$ and let $E \m W$ be a bundle with a preferred section $s_0$. Let $\Gamma^c_{(M,N)}(E)$ denote the space of sections of $E$ over $W-N$ which agree with $s_0$ on $W-M$. For $N$ empty, let $\Gamma^c_{M}(E) =\Gamma^c_{(M,N)}(E)$.
\label{definition:compactsupport}
\end{definition}

The scanning map will depend on a choice of metric on $M-N$ and function $\epsilon:M \m \R_{>0}$. Require that the ball of radius $\epsilon(m)$ around $m \in M-N$ is inside the injectivity radius. For $m \in M-N$, the map $T_mM \m M$ given by $v \m exp_m(\epsilon(m) v)$ gives a diffeomorphism between $D_mM$ and $B_{\epsilon(m)}(m)$, the ball of radius $\epsilon(m)$ around the point $m$. Smooth maps of spaces induce maps of labeled configuration spaces with the maps acting on macroscopic locations of the particles and the derivatives of the maps acting on the labels. In this way, the inverse of the above function induces a map of configuration spaces $e_m:C(M,M-B_{\epsilon (m)}(m);A) \m C(D_mM,S_mM;A)$. Using this map, we can construct the scanning map.

\begin{definition}
Let $s : C(M,N;A) \m \Gamma^c_{(M,N)}(E^A)$ be the map defined as follows. For $m \in M-N$ and $\xi \in C(M,N;A)$, let $s(\xi)(m) \in E^A_m$ be the image of $\xi$ under the maps: $$C(M,N;A) \overset{\pi}{\m} C(M,M-B_{\epsilon (m)}(m);A) \overset{e_m}{\m} C(D_mM,S_mM;A)=E^A_m.$$ Here the preferred section of $E^A$ used in the compactly supported condition is the empty section $s_0$. That is, $s_0(m) \in E_m^A = C(D_mM,S_mM;A)$ is the empty configuration. 
\end{definition}

The following lemma is an important tool in Salvatore's proof of nonabelian Poincar\e duality.

\begin{lemma}
Let $K$ be $n$-submanifold of $M$ and assume that $\pi_0(N \cap K) \m \pi_0(M)$ is onto. Then the projection map $\pi : C(M,N;A) \m C(M,K \cup N;A)$ is a quasi-fibration with fiber $C(K,K \cap N;A)$.
\label{lemma:qfib}
\end{lemma}

\begin{theorem}
If $\pi_0(N) \m \pi_0(M)$ is onto, then $s : C(M,N;A) \m \Gamma^c_{(M,N)}(E^A)$ is a weak homotopy equivalence.
\label{theorem:relscan}
\end{theorem}

\begin{proof}
The case when $N = \partial M$ is proven in \cite{Sa}. Salvatore's proof is identical to Steps 1-4 of Proposition 2 of \cite{B} after replacing labeled (without summing) configuration spaces with configuration spaces with summable labels. Steps 5-8 of Proposition 2 of \cite{B} also apply to configuration spaces with summable labels to give the above theorem. The key tool in both Salvtore and B{\"o}digheimer's proof is respectively Lemma \ref{lemma:qfib} and the analogous statement for labeled configuration spaces without summable labels.
\end{proof}

\FloatBarrier
\subsection{Homology fibrations and scanning theorems}

In this subsection, we will consider the case that $M$ is a connected open $n$-manifold which is the interior of a manifold with non-empty boundary. We will describe a stabilization procedure involving bringing points in from infinity and prove that the scanning map induces a homolology equivalence between the stabilization of $C(M;A)$ and $\Gamma_M^c(E^A)$.

In Andrade's model of topological chiral homology, the only point-set topological assumptions we made were that the $D_n$-algebras had non-degenerate basepoints. To prove non-abelian Poincar\e duality after stabilizing for configuration spaces of particles with summable labels, we need to make additional point-set topological assumptions.

\begin{definition}
A space $X$ is uniformly locally connected if there is a neighborhood $\mathcal U$ of the diagonal in $X \times X$ and a map $\lambda : \mathcal U \times [0,1] \m X$ such that $\lambda(x,y,0)=x$, $\lambda(x,y,1)=y$ and $\lambda(x,x,t)=x$.
\end{definition}

\begin{definition}
We call an $fF_n$-algebra $A$ and a smooth $n$-manifold $M$ non-pathological if $A$ is uniformly locally connected and well based and both $A$ and $C(M;A)$ are homotopy equivalent to $CW$ complexes.
\end{definition}

We will prove non-abelian Poincar\e duality after stabilizing under the assumption that the pair $(A,M)$ is non-pathological. Next we will describe the stabilization maps. For simplicity, we assume that $M$ is connected. Since we assume that $M$ is the interior of a manifold with non-empty boundary, we can find a submanifold with boundary $Q_1$, diffeomorphic to $[0,1) \times \R^{n-1}$ such that $M-Q_1$ is diffeomorphic to $M$ by a diffeomorphism isotopic to the identity. For future use, we will fix another submanifold with boundary $Q_2$ containing $Q_1$ in its interior (see Figure \ref{H}) with the same properties as $Q_1$. Let $M_i=M-Q_i$. Fix a diffeomorphism $f:M_1 \m M$ such that $f|_{M_2}=id$ and which is isotopic relative to $M_2$ to the standard inclusion $i: M_1 \m M $. Also fix a point $q \in Q_1$.

\begin{figure}
\begin{center}
\scalebox{.6}{\includegraphics{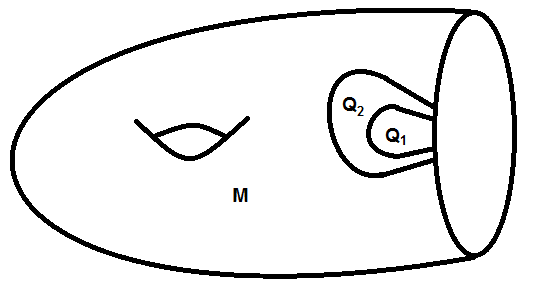}}
\caption{The subsets $Q_1$ and $Q_2$}
\label{H}
\end{center}
\end{figure}

\FloatBarrier

\begin{definition}
For $a \in A$, let $t_a : C(M;A) \m C(M;A)$ be defined as follows. The diffeomorphism $f^{-1}$ induces a map $C(M;A) \m C(M_1;A)$. Mapping a configuration of labeled points in $C(M_1;A)$ to the same configuration of points union $(q;a)$ gives a map $C(M_1;A) \m C(M;A)$. Let $t_a$ be the composition of these two maps.

\end{definition}

Let $T_{a}$ be the corresponding stabilization map for $\Gamma^c_M(E^A)$. The goal of this section is to prove the following theorem. It is nonabalian Poincar\e duality after stabilizing for configuration spaces of particles with summable labels.

\begin{theorem}
Let $M$ be a connected parallelizable $n$-manifold with $n>1$ which is the interior of a manifold with nonempty boundary. Assume that $(A,M)$ is non-pathological. Let $\{a_i\}$ be a sequence of elements of $A$ such that each connected component has an infinite number of terms of the sequence. The scanning map $s$ induces a homology equivalence between $hocolim_{t_{a_i}} C(M;A)$ and $\Gamma^c_M(E^A)$.
\label{theorem:main3}
\end{theorem}

By $hocolim_{t_{a_i}}C(M;A)$, we mean the mapping telescope of the following diagram: $$C(M;A) \overset{t_{a_{1}}}{\m} C(M;A) \overset{t_{a_{2}}}{\m} C(M;A) \overset{t_{a_{3}}}{\m} C(M;A) \ldots $$ All other similar homotopy colimits are understood to be taken over similar diagrams. 

For $M$ parallelizable,  $\Gamma^c_M(E^A) \simeq Map^c(M,B^n A)$. In particular, Theorem \ref{theorem:main3} implies Theorem \ref{theorem:main} for Salvatore's model of topological chiral homology. The rest of this section is devoted to proving Theorem \ref{theorem:main3}. Consider the following commuting diagram:

$$
\begin{array}{ccccccccl}
C(Q_2;A) & \overset{s}{\m} & \Gamma^c_{Q_2}(E^A) \\
\iota \downarrow   & & \downarrow       \\
C(M;A) & \overset{s}{\m} & \Gamma^c_{M}(E^A) \\
\pi \downarrow   & & \downarrow       \\
C(M,Q_2;A) & \overset{s}{\m} & \Gamma^c_{(M,Q_2)}(E^A). \\

\end{array}$$

Note that the right-hand side is a fiber sequence. This follows from a generalization of the fact that the functor $Map(\cdot,X)$ takes cofiber sequences to fiber sequences (for example see the proof of Proposition 2 of \cite{B}). By Theorem \ref{theorem:relscan}, the bottom row is a weak homotopy equivalence. When $A$ is grouplike, Salvatore proved that the left-hand side is a quasi-fibration \cite{Sa} and the top row is a weak equivalence. Thus, if $A$ is grouplike, $s: C(M;A) \m \Gamma^c_M(E^A)$ is a weak equivalence by the long exact sequence of homotopy groups and the five lemma. In the nongrouplike case, we will use a similar argument. First we will stabilize.  Then we will show that the left hand side is a homology fibration and that the top scanning map is a homology equivalence. Then we will use the spectral sequence comparison theorem. Since we assumed that $f$ is the identity on $M_2$, the following diagram commutes:

$$
\begin{array}{ccccccccl}
C(Q_2;A) & \overset{t_a}{\m} & C(Q_2;A) \\
\iota \downarrow   & & \iota \downarrow       \\
C(M;A) & \overset{t_a}{\m} & C(M;A) \\
\pi \downarrow   & & \pi \downarrow       \\
C(M,Q_2;A) & \overset{id}{\m} & C(M,Q_2;A) .\\

\end{array}$$

Thus, the following diagram commutes:

\begin{equation}
\begin{array}{ccccccccl}
hocolim_{t_{a_i}}C(Q_2;A) & \overset{s}{\m} & hocolim_{T_{a_i}} \Gamma^c_{Q_2}(E^A) \\
\iota \downarrow   & & \downarrow       \\
hocolim_{t_{a_i}}C(M;A) & \overset{s}{\m} & hocolim_{T_{a_i}} \Gamma^c_{M}(E^A)\\
\pi \downarrow   & & \downarrow       \\
C(M,Q_2;A) & \overset{s}{\m} & \Gamma^c_{(M,Q_2)}(E^A). \\
\label{star}
\end{array}
\end{equation}

To see that the top row of Diagram (\ref{star}) is a homology equivalence, we recall the following result from \cite{Sa}.

\begin{theorem}
The scanning map $s: C(\R^n;A) \m \Gamma^c_{\R^n}(E^A)$ is a group completion.
\end{theorem}

\begin{corollary}
For $n>1$, the scanning map  $s: C(Q_2;A) \m \Gamma^c_{Q_2}(E^A)$ induces a homology equivalence,  $s: hocolim_{t_{a_i}} C(Q_2;A) \m hocolim_{T_{a_i}} \Gamma^c_{Q_2}(E^A)$.
\end{corollary}

\begin{proof}
First note that a diffeomorpism $\R^n \m Q_2$ induces homotopy equivalences $C(\R^n;A)\m C(Q_2;A) $ and $\Gamma^c_{\R^n}(E^A) \m \Gamma^c_{Q_2}(E^A)$. The space  $C(\R^n;A)$ is a $D_n$-algebra. Thus for $n>1$, it is homotopy equivalent to a homotopy commutative monoid. Therefore, it satisfies the hypothesis of the group completion theorem \cite{MSe}. The monoid multiplication maps are homotopic to the stabilization maps $t_a : C(\R^n;A) \m C(\R^n;A)$. By the group completion theorem, we can deduce that $s: hocolim_{t_{a_i}} C(\R^n;A) \m hocolim_{t_{a_i}} \Gamma^c_{Q_2}(E^A)$ is a homology equivalence and the claim follows.

\end{proof}

The left hand side of Diagram (\ref{star}) is not always a quasi-fibration, but is a homology fibration when $(A,M)$ is non-pathological.

\begin{definition}
A map $\pi : E \m B$ is called a homology fibration if the inclusion of every fiber into the homotopy fiber is a homology equivalence. 
\end{definition}

This definition implies that the Serre spectral sequence can be used to study the homology of the total space of a homology fibration. In \cite{Mc1}, McDuff states the following sufficient condition for a map being a homology fibration.

\begin{proposition} A map $r:Y \m X$ is a homology fibration with fiber $F$ if the following 5 conditions are satisfied. Let $X=\bigcup X_k$ with each $X_i$ closed.

(i) all spaces $X_k$, $X_k-X_{k-1}, r^{-1}(X_k), r^{-1}(X_k-X_{k-1})$ have the homotopy type of $CW$ complexes;

(ii) each $X_k$ is uniformly locally connected;

(iii) each $x\in X$ has a basis of contractible neighborhoods $U$ such that the contraction of $U$ lifts to a deformation retraction of $r^{-1}(U)$ into $r^{-1}(x)$;

(iv) each $r:r^{-1}(X_k-X_{k-1}) \m X_k-X_{k-1}$ is a fibration with fiber $F$;

(v) for each $k$, there is an open subset $U_k$ of $X_k$ such that $X_{k-1} \subset U_k$, and there are homotopies $h_t : U_k \m U_k$ and $H_t : r^{-1}(U_k) \m r^{-1}(U_k)$ satisfying 

(a) $h_0=id, h_t(X_{k-1}) \subset X_{k-1}, h_1(U_k) \subset X_{k-1}$;

(b) $H_0 = id, r \circ H_t =h_t \circ r;$

(c) $H_1:r^{-1}(x) \m r^{-1}(h_1(x))$ induces an isomorphism on homology for all $x \in U_k$.


\label{theorem:DoldMcDuff}

\end{proposition}

This is analogous to a theorem of Dold and Thom in \cite{DT} involving weak equivalences and quasi-fibrations. 

\begin{lemma}

If $dim M >1$ and $(A,M)$ is non-pathological, then the map $\pi: hocolim_{t_{a_i}}C(M;A) \m C(M,Q_2;A)$ is a homology fibration with fiber $hocolim_{t_{a_i}}C(Q_2;A)$.

\end{lemma}

\begin{proof}

For simplicity of notation, we will assume that $\pi_0 (A) = \N$. The general case is similar except with more indices. Since we also are assuming that $M$ is connected, we have that $\pi_0(C(M;A)) = \pi_0(C(Q_2;A)) = \N$. The isomorphism $\pi_0 (C(M;A)) \m \pi_0(A)$ is given by multiplying together all of the elements of $A$ labeling a particular configuration of points in $M$. Let $A_k$ denote the $k$'th component of $A$. We will use Theorem \ref{theorem:DoldMcDuff} to show that $\pi$ is a homology fibration by considering the following choice of filtration on $C(M,Q_2;A)$: let $X_k$ be the subset of points where the product of the labels of points macroscopically located in $M_2$ is in $A_i$ with $i \leq k$.

Conditions $(i)$ and $(ii)$ are true since the pair $(A,M)$ is non-pathological. To see that condition $(iii)$ is satisfied, fix $x \in C(M,Q_2;A) $. Let $(m_1 \ldots m_r)$ be the macroscopic locations of the points of $x$. Let $a_i \in A$ be the label of the configuration $x$ at the point $m_i$ and let $\{\mathcal U_i\}$ be a collection of disjoint open balls such that $m_i \in \mathcal U_i \subset M_2$. Let $\mathcal A_i \subset A$ be a contractible set deformation retracting onto the point $a_i$. Let $U \subset C(\bigcup \mathcal U_i;A) \subset  C(M,Q_2;A)$ be the set of configurations of points such that the product of the labels of the points macroscopically located in $\mathcal U_i$ are in $\mathcal A_i$. Sets of this form give a basis of $C(M,Q_2;A) $ which satisfy the requirements of condition $(iii)$. Condition $(iv)$ is true since $\pi$ is in fact a trivial fibration over each $X_k-X_{k-1}$ with fiber $hocolim_{t_{a_i}}C(Q_2;A)$.

We now check condition $(v)$. Let $M_3$ be $M$ minus a closed collar neighborhood of $Q_2$. Let $U_k \subset X_k$ be the subspace of configurations of points where the product of the labels of the points macroscopically located in $M_3$ is in $A_i$ with $i<k$. Let $g_t$ be a path of diffeomorphisms of $M$ with $g_0=id$, $g_t(Q_2) \subset Q_2$ for all $t$ and $g_1(M-M_3) \subset Q_2$. The isotopy $g_t$ induces homotopies $h_t:U_k \m U_k$ and $H_t:\pi^{-1}(U_k) \m \pi^{-1}(U_k)$ by applying $g_t$ to each point in the configurations. To see that condition $(a)$ is satisfied, first note that $h_0=id$ since $g_0=id$. We have that $h_t(X_{k-1}) \subset X_{k-1}$ since $g_t(Q_2) \subset Q_2$.  Since $g_1(M-M_3) \subset Q_2$, $h_1(U_k) \subset X_{k-1}$. Condition $(b)$ follows from the fact that $f_0=id$ and $H_t$ and $h_t$ are both induced by $f_t$. The map on $hocolim_{t_{a_i}}C(Q_2;A)$ from condition $(c)$ is homotopy equivalent to one induced by a product of stabilization maps $t_{a_i}$. The stabilization maps induce homology equivalences on $hocolim_{t_{a_i}}C(Q_2;A)$ provided that $n > 1$. When $n=1$, there are two potentially non-homotopic stabilization maps $C(Q_2;A) \m C(Q_2;A)$ corresponding to adding points from the left or right. Thus $\pi$ is a homology fibration.

\end{proof}

Before we can prove Theorem \ref{theorem:main3}, we need the following lemma.

\begin{lemma}
For $dim M>1$, the action of $\pi_1(C(M,Q_2;A))$ on $H_*(hocolim_{t_{a_i}} C(Q_2;A))$ induced by the map $\pi$ is trivial. Also, the action of $\pi_1(\Gamma^c_{(M,Q_2)}(E^A))$ on $H_*(hocolim_{T_{a_i}} \Gamma^c_{Q_2}(E^A))$ is trivial.

\end{lemma}

\begin{proof}
The way $\pi_1(C(M,Q_2;A))$ acts on $H_*(hocolim_{t_{a_i}} C(Q_2;A))$ is via conjugation by the maps on homology induced by the stabilization maps. Since $C(Q_2;A)$ is homotopy commutative for $n>1$, conjugation is trivial. The same argument applies to the action of $\pi_1(\Gamma^c_{(M,Q_2)}(E^A))$ on $H_*(hocolim_{T_{a_i}} \Gamma^c_{Q_2}(E^A))$.

\end{proof}

The action could be non-trivial if we considered $Q_2$ whose interior is not of the form $\R^2 \times N$ with $N$ a connected $(n-2)$-manifold. We can now prove Theorem \ref{theorem:main3}, nonabelian Poincar\e duality after stabilizing for Salvatore's configuration space of particles with summable labels. 

\begin{proof}[Proof of Theorem \ref{theorem:main3}]
See Diagram (\ref{star}). The left row is a homology fibration and the right row is a quasi-fibration and hence also a homology fibration. The bottom row is a weak-equivalence and the top row is a homology equivalence. The fundamental groups of the bases act trivially on the fibers. Consider the Serre spectral sequences for the two homology fibrations. The scanning map induces an isomorphism between the $E_2$ pages and hence an isomorphism between the $E_{\infty}$ pages. Thus $s:hocolim_{t_{a_i}}C(M;A) \m  hocolim_{T_{a_i}} \Gamma^c_{M}(E^A)$ is a homology equivalence. Since each $T_{a_i}$ is a homotopy equivalence, $hocolim_{T_{a_i}} \Gamma^c_{M}(E^A)$ is homotopy equivalent to $\Gamma^c_{M}(E^A)$. This completes the proof.
\end{proof}

\section{A conjecture on homological stability}

There are several natural questions raised by Theorem \ref{theorem:main}. In this section, we highlight one of them and make a conjecture regarding the effects in homology of the stabilization maps $t_{a_i} : \int_M A \m \int_M A$. For simplicity of notation, assume that $M$ is connected and open and $\pi_0(A) =\N$.

\begin{definition}
Let $A$ be an $E_n$-algebra with $\pi_0(A) =\N$. Fix $b \in A_1$ and let $m_b : A_k \m A_{k+1}$ be a multiplication with $b$ map. We say that $A$ has homological stability if there is a function $r : \N \m \N$ tending to $\infty$ such that $m_{b*} : H_i(A_k) \m H_i(A_{k+1})$ is an isomorphism for $i \leq r(k)$.
\end{definition}

Examples of such $E_n$-algebras include: labeled (without summable labels) configuration spaces \cite{Se} \cite{Rw1}, rational functions \cite{Se} \cite{BM}, as well as the union of the classifying spaces of the mapping class groups of once punctured surfaces \cite{Ha}. Let $(\int_M A)_k$ denote the $k$'th component of $\int_M A$.

\begin{conjecture}
Let $M$ be an open connected parallelizable $n$-manifold and let $A$ be an $E_n$-algebra with $\pi_0(A)=\N$. Fix $b \in A_1$ and let $t_b: (\int_M A)_k \m (\int_M A)_{k+1}$ be a stabilization map described in previous sections. If $A$ has homological stability, we conjecture that there is a range of dimensions tending to infinity below which $t_b$ induces an isomorphism on homology. Moreover, this range should depend only on the homological stability range of $A$. 

\end{conjecture}

This conjecture is trivially true when $M = \R^n$. It is also true when $A = C(\R^n)X$ for a based space $X$. If true, combining this conjecture with Theorem \ref{theorem:main} would allow one to conclude that the scanning map $s:\int_M A \m Map^c(M,B^n A)$ is a homology equivalence though a range of dimensions.

\bibliography{thesis}{}

\def\cprime{$'$}
\begin{thebibliography}{McD75}

\bibitem[And10]{An}
Ricardo Andrade.
\newblock {\em From manifolds to invariants of {E}n-algebras}.
\newblock ProQuest LLC, Ann Arbor, MI, 2010.
\newblock Thesis (Ph.D.)--Massachusetts Institute of Technology.

\bibitem[BM88]{BM}
Charles~P. Boyer and Benjamin~M. Mann.
\newblock Monopoles, nonlinear {$\sigma$} models, and two-fold loop spaces.
\newblock {\em Comm. Math. Phys.}, 115(4):571--594, 1988.

\bibitem[B{\"o}d87]{B}
C.-F. B{\"o}digheimer.
\newblock Stable splittings of mapping spaces.
\newblock In {\em Algebraic topology ({S}eattle, {W}ash., 1985)}, volume 1286
  of {\em Lecture Notes in Math.}, pages 174--187. Springer, Berlin, 1987.

\bibitem[DT58]{DT}
Albrecht Dold and Ren{\'e} Thom.
\newblock Quasifaserungen und unendliche symmetrische {P}rodukte.
\newblock {\em Ann. of Math. (2)}, 67:239--281, 1958.

\bibitem[Fra11]{Fr1}
John Francis.
\newblock The tangent complex and {H}ochschild cohomology of {$E_n$}-rings.
\newblock 2011.

\bibitem[GJ94]{GJ}
Ezra Getzler and Jones John.
\newblock Operads, homotopy algebra and iterated integrals for double loop
  spaces.
\newblock {\em arXiv:hep-th/9403055}, 1994.

\bibitem[Har85]{Ha}
John~L. Harer.
\newblock Stability of the homology of the mapping class groups of orientable
  surfaces.
\newblock {\em Ann. of Math. (2)}, 121(2):215--249, 1985.

\bibitem[HM97]{HM}
Lars Hesselholt and Ib~Madsen.
\newblock On the {$K$}-theory of finite algebras over {W}itt vectors of perfect
  fields.
\newblock {\em Topology}, 36(1):29--101, 1997.

\bibitem[Kal01]{K3}
Sadok Kallel.
\newblock Spaces of particles on manifolds and generalized {P}oincar\'e
  dualities.
\newblock {\em Q. J. Math.}, 52(1):45--70, 2001.

\bibitem[Lur09]{Lu}
Jacob Lurie.
\newblock On the classification of topological field theories.
\newblock In {\em Current developments in mathematics, 2008}, pages 129--280.
  Int. Press, Somerville, MA, 2009.

\bibitem[May72]{M}
J.~P. May.
\newblock {\em The geometry of iterated loop spaces}.
\newblock Springer-Verlag, Berlin, 1972.
\newblock Lecture Notes in Mathematics, Vol. 271.

\bibitem[McD75]{Mc1}
D.~McDuff.
\newblock Configuration spaces of positive and negative particles.
\newblock {\em Topology}, 14:91--107, 1975.

\bibitem[MS76]{MSe}
D.~McDuff and G.~Segal.
\newblock Homology fibrations and the ``group-completion'' theorem.
\newblock {\em Invent. Math.}, 31(3):279--284, 1975/76.

\bibitem[RW11]{Rw1}
O.~Randal-Williams.
\newblock Homological stability for unordered configuration spaces.
\newblock 2011.

\bibitem[Sal01]{Sa}
Paolo Salvatore.
\newblock Configuration spaces with summable labels.
\newblock In {\em Cohomological methods in homotopy theory ({B}ellaterra,
  1998)}, volume 196 of {\em Progr. Math.}, pages 375--395. Birkh\"auser,
  Basel, 2001.

\bibitem[Seg68]{Se2}
Graeme Segal.
\newblock Classifying spaces and spectral sequences.
\newblock {\em Inst. Hautes \'Etudes Sci. Publ. Math.}, (34):105--112, 1968.

\bibitem[Seg79]{Se}
Graeme Segal.
\newblock The topology of spaces of rational functions.
\newblock {\em Acta Math.}, 143(1-2):39--72, 1979.

\end{thebibliography}
\bibliographystyle{alpha}

\end{document}